\documentclass{article}
\usepackage{amsmath}
\usepackage{amssymb}
\usepackage{esint}
\usepackage{geometry}
\makeatletter
 \usepackage{amsfonts}
\usepackage{color}
\usepackage{graphicx}
\providecommand{\U}[1]{\protect \rule{.1in}{.1in}}

\newtheorem{theorem}{Theorem}[section]
\newtheorem{acknowledgement}[theorem]{Acknowledgement}

\newtheorem{definition}[theorem]{Definition}

\newtheorem{lemma}[theorem]{Lemma}

\newtheorem{proposition}[theorem]{Proposition}
\newtheorem{remark}[theorem]{Remark}

\newenvironment{proof}[1][Proof]{\noindent \textbf{#1.} }{\  \rule{0.5em}{0.5em}}

\makeatother

\begin{document}

\title{Maximally distributed random fields under sublinear expectation}

\author{Xinpeng Li \& Shige Peng
\footnote{Research Center for Mathematics and Interdisciplinary Sciences, Shandong
University, 266237, Qingdao, China and School of Mathematics, Shandong University, 250100, Jinan, China.\ Email: peng@sdu.edu.cn (Shige Peng)
}}

\date{ }

\maketitle

\abstract{{This paper focuses on the maximal distribution on sublinear expectation space and introduces a new type of random fields with the maximally distributed finite-dimensional distribution. The corresponding spatial maximally distributed white noise is constructed, which includes the temporal-spatial situation as a special case due to the symmetrical independence property of maximal distribution. In addition, the stochastic integrals with respect to the spatial or temporal-spatial maximally distributed white noises are established in a quite direct way without the usual assumption of adaptability for integrand.}
}


%MSC2010:{60H05, 60H10, 60H30, 60H40, {{60G60, 60G65}}}

\section{Introduction}

In mathematics and physics, a random field is a type of parameterized family of random variables.
When the parameter is time $t\in \mathbb{R}^+$, we call it a stochastic process, or a temporal random field.
Quite often the parameter is space $x\in \mathbb{R}^d$, or time-space $(t,x)\in \mathbb{R}^+\times \mathbb{R}^d$. In this case, we call it a spatial or temporal-spatial random field.  A typical example is the electromagnetic wave dynamically
spread everywhere in our  $\mathbb{R}^3$-space or more exactly,  in $\mathbb{R}^+\times \mathbb{R}^3$-time-space. In principle,  it is impossible to know the exact state of the electromagnetic wave of our real world , namely,
it is a nontrivial random field parameterized by the time-space $(t,x)\in \mathbb{R}^+\times \mathbb{R}^3$.

Classically, a random field is defined on a given probability space $(\Omega,{\mathcal{F}},P)$.
But for the above problem, can we really get to know the probability $P$? This involves the so called
problem of uncertainty of probabilities.

Over the past few decades, non-additive probabilities or nonlinear expectations have become active domains for studying uncertainties, and  received more and more attention in many research fields, such as mathematical economics,  mathematical finance, statistics, quantum mechanics. {A} typical example of nonlinear expectation is sublinear one, which is used to model the uncertainty phenomenon characterized by a family of probability measures $\{P_\theta\}_{\theta\in\Theta}$ in which the true measure is unknown, and such sublinear expectation is usually defined by
$$\mathbb{E}[X]:=\sup_{\theta\in\Theta}E_{P_\theta}[X].$$
This notion is also known as the upper expectation in robust statistics (see Huber \cite{Huber}), or the upper prevision in the theory of imprecise probabilities (see Walley \cite{Walley}), and has the closed relation with coherent risk measures (see Artzner et al. \cite{ADEH}, Delbaen \cite{Delbaen2002}, F\"{o}llmer and Schied \cite{Fo-Sch}). A first dynamical
nonlinear expectation, called $g$-expectation was initiated by Peng \cite{Peng1997}.

{The foundation of sublinear expectation theory with a new type of $G$-Brownian
motion and  the corresponding It\^{o}'s  stochastic calculus was laid in  Peng \cite{P2007}, which keeps the rich and elegant  properties of classical probability theory except linearity of expectation.} Peng \cite{Peng2008a} initially defined the notion of independence and  identical distribution (i.i.d.)  based  on the notion of nonlinear expectation
instead of the  capacity. Based on the notion of  new notions, the most important distribution called $G$-normal distribution introduced, which can be characterized by the so-called $G$-heat equation. The notions of $G$-expectation and $G$-Brownian motion can be regarded as a nonlinear generalization of Wiener measure and classical Brownian motion. The corresponding limit theorems as well as stochastic calculus of It\^{o}'s  type under $G$-expectation are systematically developed in Peng \cite{P2010}. Besides that, there is also another important distribution, called maximal distribution. The distribution of maximally distributed random variable $X$ can be calculated simply by
$$\mathbb{E}[\varphi(X)]=\max_{v\in[-\mathbb{E}[-X],\mathbb{E}[X]]}\varphi(v), \ \ \varphi\in C_b(\mathbb{R}).$$
The law of large numbers under sublinear expectation (see Peng \cite{P2010}) shows that if $\{X_i\}_{i=1}^\infty$ is a sequence of independent and identical distributed random variables with $\lim_{c\rightarrow\infty}\mathbb{E}[(|X_1|-c)^+]=0$, then the sample average converges to maximal distribution in law, i.e.,
$$\lim_{n\rightarrow\infty}\mathbb{E}[\varphi(\frac{X_1+\cdots+X_n}{n})]=\max_{v\in[-\mathbb{E}[-X_1],\mathbb{E}[X_1]]}\varphi(v), \ \ \forall \varphi\in C_b(\mathbb{R}).$$
We note that the finite-dimensional distribution for quadratic variation process of $G$-Brownian motion is also maximal distributed.

 Recently, Ji and Peng \cite{PJ} introduced a new $G$-Gaussian random fields, which contains a type of spatial white noise as a special case. Such white noise is a natural generalization of the classical Gaussian white noise (for example, see Walsh \cite{Walsh}, Dalang \cite{Dalang} and Da Prato and Zabczyk \cite{DaP-Z}). As pointed in \cite{PJ}, the space-indexed increments {do not} satisfy the property of independence. Once the sublinear $G$-expectation degenerates to linear case, the property of independence for the space-indexed part {turns out to be true} as in the classical probability theory.

 In this paper, we introduce a {very special but also typical} random field, called maximally distributed random field, in which the finite-dimensional distribution is maximally distributed. The corresponding space-indexed white noise is also constructed. It is worth mentioning that the space-indexed increments of maximal white noise is independent, which is essentially different from the case of $G$-Gaussian white noise. {Thanks to the symmetrical independence of maximally distributed white noise, it is natural to view the temporal-spatial maximally distributed white noise as a special case of the space-indexed maximally distributed white noise. The stochastic integrals with respect to spatial and temporal-spatial maximally distributed white noises can be constructed in a quite simple way, which generalize the stochastic integral with respect to quadratic variation process of $G$-Brownian motion introduced in Peng \cite{P2010}.} {Furthermore, due to the boundedness of maximally distributed random field, the usual assumption of adaptability for integrand can be dropped.} We emphasize that the structure of maximally distributed white noise is quite simple, it can be determined by only two parameters $\underline{\mu}$ and $\overline{\mu}$, and the calculation of the corresponding finite-dimensional distribution is taking the maximum of continuous function on the domain determined by $\underline{\mu}$ and $\overline{\mu}$. The use of maximally distributed
 {random} fields for modelling purposes in applications can be explained mainly by the simplicity of their construction and analytic tractability combined with the maximal distributions of marginal which describe many real phenomena due to the law of large numbers with uncertainty.

This paper is organized as follows. In Section 2, we review basic
notions and results of nonlinear expectation theory and the notion and properties of maximal distribution. In Section 3, we first recall the general setting of random fields under nonlinear expectations, and then introduce the maximally distributed random fields. In
Section 4, we construct the spatial maximally distributed white noise and study the corresponding properties. { The properties of spatial as well as temporal-spatial maximally distributed white noise and the related stochastic integrals are established in Section 5. }

\section{Preliminaries}

In this section, we recall some basic notions and properties in the
nonlinear expectation theory. More details can be found in Denis et
al. \cite{DHP}, Hu and Peng \cite{Hu-Peng} and Peng \cite{P2007,P2008,Peng2008a,P2010a,P2010,P2019}.

Let $\Omega$ be a given nonempty set and $\mathcal{H}$ be a linear
space of real-valued functions on $\Omega$ such that if $X\in\mathcal{H}$,
then $|X|\in\mathcal{H}$. $\mathcal{H}$ can be regarded as the space
of random variables. In this paper, we consider a more convenient
assumption: if random variables $X_{1}$,$\cdots$,$X_{d}\in\mathcal{H}$,
then $\varphi(X_{1},X_{2},\cdots,X_{d})\in\mathcal{H}$ for each $\varphi\in C_{b.Lip}(\mathbb{R}^{d})$. Here $C_{b.Lip}(\mathbb{R}^{d})$
is the space of all bounded and Lipschitz functions on $\mathbb{R}^{d}$.

We call $X=(X_{1},\cdots,X_{n})$, $X_{i}\in\mathcal{H}$, $1\leq i\leq n$, an $n$-dimensional random vector, denoted by $X\in\mathcal{H}^{n}$.

\begin{definition} \label{sublinear expectation} A {nonlinear
expectation} $\hat{E}$ on $\mathcal{H}$ is a functional
$\hat{E}:\mathcal{H}\rightarrow\mathbb{R}$ satisfying
the following properties: for each $X,Y\in\mathcal{H}$,
\begin{description}
\item [{(i)}] {Monotonicity:}\quad{}$\hat{E}[X]\geq\hat{E}[Y]\ \ \text{if}\ X\geq Y$;
\item [{(ii)}] {Constant preserving:}\quad{}$\hat{E}[c]=c\ \ \text{for}\ c\in\mathbb{R}$;
\end{description}
The triplet $(\Omega,\mathcal{H},\hat{E})$ is called
a {nonlinear expectation space}.
If we further assume that
\begin{description}
\item [{(iii)}] {Sub-additivity:}\quad{}$\hat{E}[X+Y]\leq\hat{E}[X]+\hat{E}[Y]$;
\item [{(iv)}]  {Positive homogeneity:}\quad{}$\hat{E}[\lambda X]=\lambda\hat{E}[X]\ \ \text{for}\ \lambda\geq0$.
\end{description}
Then $\hat{E}$ is called a sublinear expectation, and the corresponding triplet $(\Omega,\mathcal{H},\hat{E})$ is
called a {sublinear expectation space}.
\end{definition}

Let $(\Omega,\mathcal{H},\hat{E})$
be a nonlinear (resp., sublinear) expectation space. For each given
$n$-dimensional random vector $X$, we define a functional on $C_{b.Lip}(\mathbb{R}^{n})$
by
\[
\mathbb{{F}}_{X}[\varphi]:=\hat{E}[\varphi(X)],\text{ for each }\varphi\in C_{b.Lip}(\mathbb{R}^{n}).
\]
$\mathbb{{F}}_{X}$ is called the distribution of $X$.
It is easily seen that $(\mathbb{R}^{n},C_{b.Lip}(\mathbb{R}^{n}),\mathbb{{F}}_{X})$
forms a nonlinear (resp., sublinear) expectation space. {If $\mathbb{{F}}_{X}$ is not a linear functional on $C_{b.Lip}(\mathbb{R}^{n})$, we say $X$ has distributional uncertainty.}

\begin{definition}\label{id}
Two $n$-dimensional random vectors
$X_{1}$ and $X_{2}$ defined on nonlinear expectation spaces $(\Omega_{1},\mathcal{H}_{1},\hat{E}_{1})$
and $(\Omega_{2},\mathcal{H}_{2},\hat{E}_{2})$ respectively,
are called {identically distributed}, denoted by $X_{1}\overset{d}{=}X_{2}$,
if $\mathbb{{F}}_{X_{1}}=\mathbb{{F}}_{X_{2}}$, i.e.,
\[
\hat{E}_{1}[\varphi(X_{1})]=\hat{E}_{2}[\varphi(X_{2})], \ \ \forall\varphi\in C_{b.Lip}(\mathbb{R}^{n}).
\]
\end{definition}

\begin{definition} Let $(\Omega,\mathcal{H},\hat{E})$
be a nonlinear expectation space. An $n$-dimensional random vector
$Y$ is said to be {independent} from another $m$-dimensional
random vector $X$ under the expectation $\hat{E}$
if, for each test function $\varphi\in C_{b.Lip}(\mathbb{R}^{m+n})$,
we have
\[
\hat{E}[\varphi(X,Y)]=\hat{E}[\hat{E}[\varphi(x,Y)]_{x=X}].
\]
\end{definition}

\begin{remark}
Peng \cite{Peng2008a} (see also Peng \cite{P2010}) introduced the notions of the distribution
and the independence of random variables under
a nonlinear expectation, which play a crucially important role in the nonlinear expectation theory.
\end{remark}

For simplicity, the sequence $\{X_i\}_{i=1}^n$ is called independence if $X_{i+1}$ is independent from $(X_1,\cdots,X_i)$ for $i=1,2,\cdots,n-1$. Let $\bar{X}$ and $X$ be two $n$-dimensional random vectors on
$(\Omega,\mathcal{H},\hat{E})$. $\bar{X}$ is called an
independent copy of $X$, if $\bar{X}\overset{d}{=}X$ and $\bar{X}$
is independent from $X$.

\begin{remark}\label{remark} It is important to note that ``$Y$ is independent
from $X$'' does not imply that ``$X$ is independent from $Y$''
(see Peng \cite{P2010}). \end{remark}

In this paper, we focus on an important distribution on sublinear expectation
space $(\Omega,\mathcal{H},\hat{E})$, called maximal distribution.

\begin{definition} An $n$-dimensional random vector $X=(X_{1},\cdots,X_{n})$
on a sublinear expectation space $(\Omega,\mathcal{H},\hat{E})$
is said to be {maximally distributed}, if there exists a bounded and
closed convex subset $\Lambda\subset\mathbb{R}^{n}$ such that, for every continuous function $\varphi\in C(\mathbb{R}^n)$,
\[
\hat{E}[\varphi(X)]=\max_{x\in\Lambda}\varphi(x).
\]
\end{definition}
{
\begin{remark}
Here $\Lambda$ characterizes the uncertainty of $X$. It is easy to check that this maximally distributed random vector $X$ satisfies
$$X+\bar{X}\overset{d}{=}2X,$$
where $\bar{X}$ is an independent copy of $X$. Conversely, suppose a random variable $X$ satisfying $X+\bar{X}\overset{d}{=}2X$, if we further assume the uniform convergence condition
$\lim_{c\rightarrow\infty}\hat{E}[(|X|-c)^+]=0$ holds, then we can deduce that $X$ is maximally distributed by the law of large numbers (see Peng \cite{P2010}).  An interesting problem is that is $X$  still maximally distributed without such uniform convergence condition? We emphasize that the law of large numbers does not hold in this case, a counterexample can be found in Li and Zong \cite{LZ}.
\end{remark}}

\begin{proposition} \label{G heat equation} Let $g(p)=\max_{v\in\Lambda}v\cdot p$
be given. Then an $n$-dimensional random variable is maximally distributed
if and only if for each $\varphi\in C(\mathbb{R}^{n})$, the following
function
\begin{equation}
u(t,x):=\hat{E}[\varphi(x+tX)]=\max_{v\in\Lambda}\varphi(x+tv),\ (t,x)\in\lbrack0,\infty)\times\mathbb{R}^{n}\label{G solution}
\end{equation}
is the unique viscosity solution of the the
following nonlinear partial differential equation
\begin{equation}
\partial_{t}u-g(D_{x}u)=0,\ \ u|_{t=0}=\varphi(x).\label{G equation}
\end{equation}
\end{proposition}

This property implies that, each sublinear function $g$ on $\mathbb{R}^{n}$
determines uniquely a maximal distribution.
The following property is easy to check.

\begin{proposition} Let $X$ be an $n$-dimensional maximally
distributed random vector characterized by its generating function
\[
g(p):=\hat{E}[X\cdot p],\,\,\,p\in\mathbb{R}^{n}.
\]
Then, for any function $\psi\in C(\mathbb{R}^{n})$, $Y=\psi(X)$
is also an $\mathbb{R}$-valued maximally distributed random variable:
\[
\mathbb{E}[\varphi(Y)]=\max_{v\in[\underline{{\rho}},\overline{\rho}]}\varphi(v),\qquad\overline{\rho}=\max_{\gamma\in\Lambda}\psi(\gamma),\quad\underline{{\rho}}=\min_{\gamma\in\Lambda}\psi(\gamma).
\]
\end{proposition}

\begin{proposition}\label{p1}
Let $X=(X_{1},\cdots,X_{n})$ be an $n$-dimensional maximal distribution
on a sublinear expectation space $(\Omega,\mathcal{H},\hat{E}).$ { If} the corresponding generating function satisfies, for all $p=(p_1,\cdots,p_n)\in\mathbb{R}^{n},$
$$g(p)=\hat{E}[X_1 p_1+\cdots+X_n p_n]=\hat{E}[X_1 p_1]+\cdots+\hat{E}[X_np_n],$$
then $\{X_i\}_{i=1}^n$ is a sequence of independent maximally distributed random variables.

Moreover, for any permutation $\pi$ of $\{1,2,\cdots,n\}$, the sequence $\{X_{\pi(i)}\}_{i=1}^n$ is also independent.
\end{proposition}
\begin{proof}
For $i=1,\cdots, n$, we denote $\overline{\mu}_i=\hat{E}[X_i]$ and $\underline{\mu}_i=-\hat{E}[-X_i]$. Since
\begin{align*}
g(p)&=\hat{E}[X_1\cdot p_1+\cdots+X_n\cdot p_n]=\hat{E}[X_1\cdot p_1]+\hat{E}[X_2\cdot p_2]+\cdots+\hat{E}[X_n\cdot p_n]\\
&=\sum_{i=1}^n\max_{v_i\in[\underline{\mu}_i,\overline{\mu}_i]}p_iv_i=\max_{(v_1,\cdots,v_n)\in\otimes_{i=1}^n[\underline{\mu}_i,\overline{\mu}_i]}(p_1v_1+\cdots+p_nv_n),
\end{align*}
it follows Proposition \ref{G heat equation} that $(X_1,\cdots,X_n)$ is an $n$-dimensional maximally distributed random vector such that, $\forall \varphi\in C(\mathbb{R}^n)$,
$$\hat{E}[\varphi(X_1,\cdots,X_n)]=\max_{(v_1,\cdots,v_n)\in\otimes_{i=1}^n[\underline{\mu}_i,\overline{\mu}_i]}\varphi(v_1,\cdots,v_n).$$
It is easy to check that $\{X_i\}_{i=1}^n$ is independent, and so does the permuted sequence $\{X_{\pi(i)}\}_{i=1}^n$.
\end{proof}
\begin{remark}
The independence of maximally distributed random variables is {symmetrical}. But, as discussed in Remark~\ref{remark}, under a
sublinear expectation, $X$ is independent from $Y$ does not automatically imply that $Y$ is also
independent from $X$. In fact, Hu and Li \cite{HL} proved that, if $X$ is independent  from $Y$, and $Y$ is also independent from $X$, {and both of $X$ and $Y$ have distributional uncertainty, then $(X,Y)$ must be maximally distributed.}
\end{remark}

\section{Maximally distributed random fields}

In this section, we first recall the general setting of random fields defined on a nonlinear expectation space introduced by Ji and Peng \cite{PJ}.

\begin{definition} Under a given nonlinear expectation space  $(\Omega,\mathcal{H},\hat{E})$, a collection of $m$-dimensional random vectors $W=(W_{\gamma})_{\gamma\in\Gamma}$  is called an $m$-dimensional random field indexed by $\Gamma$, if for each   $\gamma\in \Gamma$,  $W_{\gamma}\in\mathcal{H}^{m}$.
\end{definition}

In order to introduce the notion of finite-dimensional distribution of a random field $W$, we denote the family of all sets of finite indices by
\[
\mathcal{J}_{\Gamma}:=\{\underline{\gamma}=(\gamma_{1},\cdots,\gamma_{n}):\ \forall n\in\mathbb{N},\ \gamma_{1},\cdots,\gamma_{n}\in\Gamma,\  {\gamma_{i}\neq\gamma_{j}\ \text{if}\ i\neq j}\}.
\]

\begin{definition} Let $(W_{\gamma})_{\gamma\in\Gamma}$ be an $m$-dimensional
 random field defined on a nonlinear expectation space $(\Omega,\mathcal{H},\hat{E})$.
For each $\underline{\gamma}=(\gamma_{1},\cdots,\gamma_{n})\in\mathcal{J}_{\Gamma}$
and the corresponding random vector $W_{\underline{\gamma}}=(W_{\gamma_{1}},\cdots,W_{\gamma_{n}})$,
we define a functional on $C_{b.Lip}(\mathbb{R}^{n\times m})$ by
\[
\mathbb{F}_{\underline{\gamma}}^{W}[\varphi]=\hat{E}[\varphi(W_{\underline{\gamma}})]
\]
The collection   $(\mathbb{F}_{\underline{\gamma}}^{W}[\varphi])_{\underline{\gamma}\in\mathcal{J}_{\Gamma}}$
is called the  family of {finite-dimensional distributions} of $(W_{\gamma})_{\gamma\in\Gamma}$.
\end{definition}
It is clear that, for each $\underline{\gamma}\in\mathcal{J}_{\Gamma}$, the triple $(\mathbb{R}^{n\times m},C_{b.Lip}(\mathbb{R}^{n\times m}),
\mathbb{F}_{\underline{\gamma}}^{W})$
constitutes a nonlinear expectation space.

Let $(W_{\gamma}^{(1)})_{\gamma\in\Gamma}$ and $({W}_{\gamma}^{(2)})_{\gamma\in\Gamma}$ be two $m$-dimensional random fields
 defined on nonlinear expectation spaces $(\Omega_{1},\mathcal{H}_{1},\hat{E}_{1})$
and $(\Omega_{2},\mathcal{H}_{2},\hat{E}_{2})$ respectively.
They are said to be {identically distributed}, denoted by $(W_{\gamma}^{(1)})_{\gamma\in\Gamma}\overset{d}{=}({W}_{\gamma}^{(2)})_{\gamma\in\Gamma}$,
or simply $W^{(1)}\overset{d}{=}W^{(2)}$, if for each $\underline{\gamma}=(\gamma_{1},\cdots,\gamma_{n})\in\mathcal{J}_{\Gamma}$,
\[
\hat{E}_{1}[\varphi(W_{\underline{\gamma}}^{(1)})]=\hat{E}_{2}[\varphi({W}_{\underline{\gamma}}^{(2)})], \ \ \forall \varphi\in C_{b.Lip}(\mathbb{R}^{n\times m}).
\]

For any given $m$-dimensional  random field $W=(W_{\gamma})_{\gamma\in\Gamma}$,
the family of its finite-dimensional distributions satisfies the following
properties of consistency:
\begin{description}
\item [{{(1)} Compatibility:}] For each $(\gamma_{1},\cdots,\gamma_{n},\gamma_{n+1})\in\mathcal{J}_{\Gamma}$
and $\varphi\in C_{b.Lip}(\mathbb{R}^{n\times m}),$
\begin{equation}\mathbb{F}_{\gamma_{1},\cdots,\gamma_{n}}^{W}[\varphi]=\mathbb{F}_{\gamma_{1},\cdots,\gamma_{n},\gamma_{n+1}}^{W}[\widetilde{\varphi}],\label{F consistent 1}
\end{equation}
where the function $\widetilde{\varphi}$ is a function on $\mathbb{R}^{(n+1)\times m}$ defined  for any $y_{1},\cdots,y_{n},y_{n+1}\in\mathbb{R}^{m}$,
$$\widetilde{\varphi}(y_{1},\cdots,y_{n},y_{n+1})=\varphi(y_{1},\cdots,y_{n});$$

\item [{{(2)} Symmetry:}] For each $(\gamma_{1},\cdots,\gamma_{n})\in\mathcal{J}_{\Gamma}$,
$\varphi\in C_{b.Lip}(\mathbb{R}^{n\times m})$ and each permutation
$\pi$ of $\{1,\cdots,n\}$,
\begin{equation}\mathbb{F}_{\gamma_{\pi(1)},\cdots,\gamma_{\pi(n)}}^{W}[\varphi]=\mathbb{F}_{\gamma_{1},\cdots,\gamma_{n}}^{W}[\varphi_{\pi}]\label{F consistent 2}
\end{equation}
where we denote $\varphi_{\pi}(y_{1},\cdots,y_{n})=\varphi(y_{\pi(1)},\cdots,y_{\pi(n)})$,
for $y_{1},\cdots,y_{n}\in\mathbb{R}^{m}$.
\end{description}
The following theorem generalizes the classical Kolmogorov's existence theorem to the situation
of sublinear expectation space, which is a variant of Theorem 3.8
in Peng \cite{Peng2011}. The proof can be founded in Ji and Peng \cite{PJ}.

\begin{theorem} \label{Kolmogorov existence} Let $\{\mathbb{F}_{\underline{\gamma}},\underline{\gamma}\in\mathcal{J}_{\Gamma}\}$
be a family of finite-dimensional distributions satisfying the compatibility condition \eqref{F consistent 1} and
the symmetry condition \eqref{F consistent 2}. Then there exists
an $m$-dimensional  random field $W=(W_{\gamma})_{\gamma\in\Gamma}$
defined on a nonlinear expectation space $(\Omega,\mathcal{H},\hat{E})$
whose family of finite-dimensional distributions coincides with $\{\mathbb{F}_{\underline{\gamma}},\underline{\gamma}\in\mathcal{J}_{\Gamma}\}$.
Moreover, if we assume that each $\mathbb{F}_{\underline{\gamma}}$
in $\{\mathbb{F}_{\underline{\gamma}},\underline{\gamma}\in\mathcal{J}_{\Gamma}\}$
is sublinear, then the corresponding expectation $\hat{E}$
on the space of random variables $(\Omega,\mathcal{H})$ is also sublinear.
\end{theorem}

Now we consider a new random fields under a sublinear expectation space.

\begin{definition} Let $(W_{\gamma})_{\gamma\in\Gamma}$ be an $m$-dimensional
random field, indexed by $\Gamma$, defined on a sublinear expectation space $(\Omega,\mathcal{H},\hat{E})$.
$(W_{\gamma})_{\gamma\in\Gamma}$ is called a {maximally distributed random field} if for each $\underline{\gamma}=(\gamma_{1},\cdots,\gamma_{n})\in\mathcal{J}_{\Gamma}$,
the following $(n\times m)$-dimensional random vector
\begin{align*}
W_{\underline{\gamma}}=&(W_{\gamma_{1}},\cdots,W_{\gamma_{n}}) \\
=&(W_{\gamma_{1}}^{(1)},\cdots W_{\gamma_{1}}^{(m)},\cdots,W_{\gamma_{n}}^{(1)},\cdots,W_{\gamma_{n}}^{(m)}),\,\,W_{\gamma_{i}}^{(j)}\in\mathcal{H},
%&\,i=1,\cdots,n,\,j=1,\cdots,m,
\end{align*}
is maximally distributed.
\end{definition}

For each $\underline{\gamma}=(\gamma_{1},\cdots,\gamma_{n})\in\mathcal{J}_{\Gamma}$,
we define
\[
g^{W}_{\underline{\gamma}}(p)=\hat{E}[W_{\underline{\gamma}}\cdot p],\,\,\,\,p\in\mathbb{R}^{n\times m},
\]
Then $(g^{W}_{\underline{\gamma}})_{\underline{\gamma}\in\mathcal{J}_{\Gamma}}$
constitutes a family of {sublinear} functions:
\[
g^{W}_{\underline{\gamma}}:\mathbb{R}^{n\times m}\mapsto\mathbb{R},\,\,\,\,\underline{\gamma}=(\gamma_{1},\cdots,\gamma_{n}),\,\,\gamma_{i}\in\Gamma,\,\,1\leq i\leq n,\,\,n\in\mathbb{N},
\]
which satisfies the properties of consistency in the following sense:
\begin{description}
\item [{{(1)} Compatibility:}] For any $(\gamma_{1},\cdots,\gamma_{n},\gamma_{n+1})\in\mathcal{J}_{\Gamma}$
and $p=(p_{i})_{i=1}^{n\times m}\in\mathbb{R}^{n\times m}$,
\begin{equation}
g^{W}_{\gamma_{1},\cdots,\gamma_{n},\gamma_{n+1}}(\bar{p})=g^{W}_{{\gamma_{1}},\cdots,W_{\gamma_{n}}}(p),\label{G consistent 1}
\end{equation}
where $\bar{p}=\left(\begin{array}{c}
p\\
0
\end{array}\right)\in\mathbb{R}^{(n+1)\times m}$;
\item [{{(2)} Symmetry:}] For any permutation $\pi$ of $\{1,\cdots,n\}$,
\begin{equation}
g^{W}_{\gamma_{\pi(1)},\cdots,\gamma_{\pi(n)}}(p)=g^{W}_{\gamma_{1},\cdots,\gamma_{n}}(\pi^{-1}(p)),\label{G consistent 2}
\end{equation}
where
$\pi^{-1}(p)=(p^{(1)},\ldots,p^{(n)}),$
$$p^{(i)}=(p_{(\pi^{-1}(i)-1)m+1},\ldots,p_{(\pi^{-1}(i)-1)m+m})\ , \ 1\leq i\leq n.$$
\end{description}

If the above
type of family of sublinear functions $(g_{\underline{\gamma}})_{\underline{\gamma}\in\mathcal{J}_{\Gamma}}$
is given, following the construction procedure in the proof of Theorem 3.5 in Ji and Peng \cite{PJ}, we can construct a maximally
distributed random field on sublinear expectation space.

\begin{theorem} \label{existence of G.R.F.} Let $(g_{\underline{\gamma}})_{\underline{\gamma}\in\mathcal{J}_{\Gamma}}$
be a family of real-valued functions such that, for each $\underline{\gamma}=(\gamma_{1},\cdots,\gamma_{n})\in\mathcal{J}_{\Gamma}$,
the real function $g_{\underline{\gamma}}$ is defined on $\mathbb{R}^{n\times m}\mapsto\mathbb{R}$
and satisfies the sub-linearity.
Moreover, this family $(g_{\underline{\gamma}})_{\underline{\gamma}\in\mathcal{J}_{\Gamma}}$
satisfies the compatibility condition \eqref{G consistent 1}
and symmetry condition \eqref{G consistent 2}. Then there exists
an $m$-dimensional  maximally distributed random field $(W_{\gamma})_{\gamma\in\Gamma}$
on a sublinear expectation space $(\Omega,\mathcal{H},\hat{E})$
such that for each $\underline{\gamma}=(\gamma_{1},\cdots,\gamma_{n})\in\mathcal{J}_{\Gamma}$,
$W_{\underline{\gamma}}=(W_{\gamma_{1}},\cdots,W_{\gamma_{n}})$ is
maximally distributed with generating function
\[
g^{W}_{{\underline{\gamma}}}(p)=\hat{E}[W_{\underline{\gamma}}\cdot p]=g_{\underline{\gamma}}(p),\,\,\text{ for any }\,p\in\mathbb{R}^{n\times m}.
\]
Furthermore, if there exists another maximally distributed random
field $({\bar{W}}_{\gamma})_{\gamma\in\Gamma}$, with the same index
set $\Gamma$, defined on a sublinear expectation space $(\bar{\Omega},\bar{\mathcal{H}},\bar{E})$
such that for each $\underline{\gamma}=(\gamma_{1},\cdots,\gamma_{n})\in\mathcal{J}_{\Gamma}$,
$\bar{W}_{\underline{\gamma}}$ is maximally distributed with
the same generating function $g_{\underline{\gamma}}$, namely,
\[
\bar{E}[\bar{W}_{\underline{\gamma}}\cdot p]=g_{\underline{\gamma}}(p)\,\,\text{ for any }\ p\in\mathbb{R}^{n\times m},
\]
then we have $W\overset{d}{=}\bar{W}$. \end{theorem}

\section{Maximally distributed white noise}

In this section, we formulate a new type of maximally distributed
white noise on $\mathbb{R}^d$.

Given sublinear expectation space $\Omega,\mathcal{H},\hat{E}$, let $\mathbb{L}^{p}(\Omega)$ be the completion
of $\mathcal{H}$ under the Banach norm $\|X\|:=\hat{E}[|X|^p]^{\frac{1}{p}}$.
For any $X,Y\in\mathbb{L}^{1}(\Omega)$,
we say that $X=Y$ if $\hat{E}[|X-Y|]=0$.
As shown in Chapter 1 of Peng \cite{P2010},
$\hat{E}$ can be continuously extended to the mapping from
$\mathbb{L}^{1}(\Omega)$ to $\mathbb{R}$ and properties (i)-(iv)
of Definition \ref{sublinear expectation} still hold. Moreover, $(\Omega,\mathbb{L}^{1}(\Omega),\hat{E})$
also forms a sublinear expectation space, which is called the complete
sublinear expectation space.

\begin{definition} \label{Gwhitenoise} Let $(\Omega,\mathbb{L}^{1}(\Omega),\hat{E})$
be a complete sublinear expectation space and $\Gamma=\mathcal{B}_{0}(\mathbb{R}^{d}):=\{A\in\mathcal{B}(\mathbb{R}^{d}),\lambda_{A}<\infty\}$,
where $\lambda_{A}$ denotes the Lebesgue measure of $A\in\mathcal{B}(\mathbb{R}^{d})$.
Let $g:\mathbb{R}\mapsto\mathbb{R}$ be a given sublinear function, i.e.,
$$g(p)=\overline{\mu}p^+-\underline{\mu}p^-, \ \ -\infty<\underline{\mu}\leq\overline{\mu}<+\infty.$$
A  random field $W=\{W_{A}\}_{A\in\Gamma}$
is called a one-dimensional {maximally distributed white noise}
if
\begin{description}
\item [{(i)}] For each $A_{1},\cdots,A_{n}\in\Gamma$, $(W_{A_{1}},\cdots,W_{A_{n}})$
is a $\mathbb{R}^{n}$-maximally distributed random vector under $\hat{E}$, and for each $A\in\Gamma$,
\begin{align}
\hat{E}[W_{A}\cdot p]=g(p)\lambda_{A},\quad p\in\mathbb{R}.\label{Gwn1}
\end{align}

\item [{(ii)}] Let $A_{1},A_{2}, \cdots, A_n$ be in $\Gamma$ and mutually disjoint, then  $\{W_{A_i}\}_{i=1}^n$ are independent sequence, and
\begin{align}
 & W_{A_{1}\cup A_{2}\cup\cdots\cup A_n}=W_{A_{1}}+W_{A_{2}}+{\cdots+W_{A_n}}. \label{Gwn2}
\end{align}
\end{description}
\end{definition}
{\begin{remark}\label{rem1}
For each $A\in\Gamma$, we can restrict that $W_A$ takes values in $[\lambda_A\underline{\mu},\lambda_A\overline{\mu}]$. Indeed, let
$$d_A(x):=\min_{y\in[\lambda_A\underline{\mu},\lambda_A\overline{\mu}]}\{|x-y|\},$$
by the definition of maximal distribution,
$$\hat{E}[d_A(W_A)]=\max_{v\in[\lambda_A\underline{\mu},\lambda_A\overline{\mu}]}\min_{y\in[\lambda_A\underline{\mu},\lambda_A\overline{\mu}]}\{|v-y|\}=0,$$
which implies that $d_A(W_A)=0$.
\end{remark}
}

We can construct a spatial maximal white noise satisfying Definition \ref{Gwhitenoise} in the following way.

For each $\underline{\gamma}=(A_1,\cdots, A_n)\in \mathcal{J}_\Gamma$, $\Gamma=\mathcal{B}_0(\mathbb{R}^d),$ consider the mapping $g_{\underline{\gamma}}(\cdot):\mathbb{R}^n\rightarrow\mathbb{R}$ defined as follows:

\begin{align}\label{eqg}
g_{\underline{\gamma}}(p){:=}\sum_{k\in\{0,1\}^n}g(k\cdot p)\lambda_{B(k)}, \ \ \ p\in\mathbb{R}^n,
\end{align}
where $k=(k_1,\cdots,k_n)\in\{0,1\}^n$, and  $B(k)=\cap_{j=1}^n B_j$, with
\[
B_{j}=\left\{\begin{array}{cc}
A_{j}\,\,\, & \text{if }k_{j}=1,\\
A_{j}^{c} & \text{if }\;k_{j}=0.
\end{array}\right.
\]

For example, given $A_{1},A_{2},A_{3}\in\Gamma$ and $p=(p_{1},p_{2},p_{3})\in\mathbb{R}^{3}$,
\begin{align*}
&g_{A_{1},A_{2},A_{3}}(p)=  g(p_{1}+p_{2}+p_{3})\lambda_{A_{1}\cap A_{2}\cap A_{3}}\\
 & +g(p_{1}+p_{2})\lambda_{A_{1}\cap A_{2}\cap A_{3}^{c}}+g(p_{2}+p_{3})\lambda_{A_{1}^{c}\cap A_{2}\cap A_{3}}+g(p_{1}+p_{3})\lambda_{A_{1}\cap A_{2}^{c}\cap A_{3}}\\
 & +g(p_{1})\lambda_{A_{1}\cap A_{2}^{c}\cap A_{3}^{c}}+g(p_{2})\lambda_{A_{1}^{c}\cap A_{2}\cap A_{3}^{c}}+g(p_{3})\lambda_{A_{1}^{c}\cap A_{2}^{c}\cap A_{3}}.
\end{align*}

Obviously, for each $\underline{\gamma}=(A_1,\cdots, A_n)\subset\Gamma$, $g_{\underline{\gamma}}(\cdot)$ defined by (\ref{eqg}) is a sublinear function defined on $\mathbb{R}^n$ due to the sub-linearity of function $g(\cdot)$. The following property shows that the consistency conditions (\ref{G consistent 1}) and (\ref{G consistent 2}) also hold for $\{g_{\underline{\gamma}}\}_{\gamma\in\mathcal{J}_\Gamma}$.

\begin{proposition}\label{G-w.n.}
The family $\{g_{\underline{\gamma}}\}_{\gamma\in\mathcal{J}_\Gamma}$ defined by (\ref{eqg}) satisfies the consistency conditions (\ref{G consistent 1}) and (\ref{G consistent 2}).
\end{proposition}
\begin{proof}
For compatibility (\ref{G consistent 1}), given $A_1,\cdots, A_n, A_{n+1}\in \Gamma$ and  $\bar{p}^T=(p^T,0)\in\mathbb{R}^{n+1}$, we have
\begin{align*}
g_{A_1,\cdots,A_{n+1}}(\bar{p})&=\sum_{k\in\{0,1\}^{n+1}}g(k\cdot\bar{p})\lambda_{B(k)}\\
&=\sum_{k'\in\{0,1\}^n}g(k'\cdot p)(\lambda_{B(k')\cap A_{n+1}}+\lambda_{B(k')\cap A_{n+1}^c})\\
&=\sum_{k'\in\{0,1\}^n}g(k'\cdot p)\lambda_{B(k')}=g_{A_1,\cdots,A_n}(p).
\end{align*}
The symmetry (\ref{G consistent 2}) can be easily verified since the operators $k\cdot p$ and $B(k)=\cap_{j=1}^nB_j$ are also symmetry.
\end{proof}

Now we present the existence of the maximally distributed white noises under the sublinear expectation.

\begin{theorem} \label{existence of G.W.N.} For each given sublinear function
\[
g(p)=\max_{\mu\in[\underline{\mu},\overline{\mu}]}(\mu\cdot p)=\overline{\mu}p^+-\underline{\mu}p^-,\ p\in\mathbb{R},
\]
there exists a one-dimensional maximally distributed random
field $(W_{\gamma})_{\gamma\in\Gamma}$ on a sublinear expectation
space $(\Omega,\mathbb{L}^1(\Omega),\hat{E})$ such that, for
each $\underline{\gamma}=(A_{1},\cdots,A_{n})\in\mathcal{J}_{\Gamma}$,
$W_{\underline{\gamma}}=(W_{A_{1}},\cdots,W_{A_{n}})$ is maximally
distributed.

Furthermore, $(W_{\gamma})_{\gamma\in\Gamma}$ is a spatial maximally distributed white
noise under $(\Omega,\mathbb{L}^{1}(\Omega),\hat{E})$, namely, conditions (i)
and (ii) of Definition~\ref{Gwhitenoise} are satisfied.

If $(\bar{W}_{\gamma})_{\gamma\in\Gamma}$ is another maximally distributed white noise
with the same sublinear function $g$ in \eqref{eqg}, then ${\bar{W}}\overset{d}{=}{W}$.\end{theorem}

\begin{proof} Thanks to Proposition \ref{G-w.n.}
and Theorem \ref{existence of G.R.F.},  the existence and uniqueness
of the maximally distributed random field $W$ in a sublinear expectation
space $(\Omega,\mathbb{L}^1(\Omega),\hat{E})$ with the family of
generating functions defined by \eqref{eqg} hold. We only need to verify that the maximally distributed random field $W$ satisfies conditions
(i) and (ii) of Definition~\ref{Gwhitenoise}.

For each $A\in\Gamma$, $\hat{E}[W_A\cdot p]=g(p)\lambda_A$ by Theorem \ref{existence of G.R.F.} and \eqref{eqg}, thus (i) of Definition~\ref{Gwhitenoise} holds.

We note that if $\{A_i\}_{i=1}^n$ are mutually disjoint, then for $p=(p_1,\cdots,p_n)\in\mathbb{R}^n$, by \eqref{eqg}, we have
$$\hat{E}[p_1W_{A_1}+\cdots+p_n W_{A_n}]=g(p_1)\lambda_{A_1}+\cdots+g(p_n)\lambda_{A_n},$$
thus the independence of $\{W_{A_i}\}_{i=1}^n$ can be implied by Proposition \ref{p1}.

In order to prove \eqref{Gwn2}, we only consider the case of two disjoint sets. Suppose that
\begin{equation*}
A_{1}\cap A_{2}=\emptyset,\quad A_{3}=A_{1}\cup A_{2},\label{eq:a12}
\end{equation*}
an easy computation of \eqref{eqg} shows that
\begin{align*}
g_{A_{1},A_{2},A_{3}}(p)=&g(p_{1}+p_{3})\lambda_{A_{1}}+g(p_{2}+p_{3})\lambda_{A_{2}}\\
=&\max_{v_{1}\in[\underline{{\mu}}\lambda_{A_{1}},\overline{\mu}\lambda_{A_{1}}]}\max_{v_{2}\in[\underline{{\mu}}\lambda_{A_{2}},\overline{\mu}\lambda_{A_{2}}]}\max_{v_{3}=v_{1}+v_{2}}(p_{1}\cdot v_{1}+p_{2}\cdot v_{2}+p_{3}\cdot v_{3}).
\end{align*}

Thus, for each $\varphi\in C(\mathbb{R}^{3})$,
\[
\hat{E}[\varphi(W_{A_{1}},W_{A_{2}},W_{A_{3}})]=\max_{v_{1}\in[\underline{{\mu}}\lambda_{A_{1}},\overline{\mu}\lambda_{A_{1}}]}\max_{v_{2}\in[\underline{{\mu}}\lambda_{A_{2}},\overline{\mu}\lambda_{A_{2}}]}\max_{v_{3}=v_{1}+v_{2}}\varphi(v_{1},v_{2},v_{3}).
\]
In particular, we set $\varphi(v_{1},v_{2},v_{3})=|v_{1}+v_{2}-v_{3}|$,
it follows that
\[
\hat{E}[|W_{A_{1}}+W_{A_{2}}-W_{A_{1}\cup A_{2}}|]=0.
\]
which implies that
\[
W_{A_{1}\cup A_{2}}=W_{A_{1}}+W_{A_{2}}.
\]

Finally, (ii) of Definition~\ref{Gwhitenoise} holds.
\end{proof}
{\begin{remark}
The finite-dimensional distribution of maximally distributed whiten noise can be uniquely determined by two parameters $\overline{\mu}$ and $\underline{\mu}$, which can be simply calculated by taking the maximum of the continuous function over the domain determined by $\overline{\mu}$ and $\underline{\mu}$.
\end{remark}}
Similar to the invariant property of $G$-Gaussian white noise introduced in Ji and Peng \cite{PJ}, it also holds for maximally distributed white noise due to the well-known invariance of the Lebesgue measure under rotation and translation.

\begin{proposition}
For each $p\in\mathbb{R}^d$ and $O\in\mathbb{O}(d):=\{O\in\mathbb{R}^{d\times d}:O^T=O^{-1}\}$, we set
$$T_{p,O}(A)=O\cdot A+p,\ \ \ A\in\Gamma.$$
Then, for each $A_1,\cdots, A_n\in\Gamma$,
$$(W_{A_1},\cdots,W_{A_n})\overset{d}{=}(W_{T_{p,O}(A_1)},\cdots,W_{T_{p,O}(A_n)}).$$
\end{proposition}

\section{Spatial and temporal maximally distributed white noise and related stochastic integral }

{
In Ji and Peng \cite{PJ}, we see that a spatial $G$-white noise is essentially different from the temporal case or  the temporal-spatial case, since there is no independence property for the spatial $G$-white noise. But for the maximally distributed white noise, spatial or temporal-spatial maximally distributed white noise has the independence property due to the symmetrical independence for maximal distribution.

Combining symmetrical independence and boundedness properties of maximal distribution, the
integrand random fields can be largely extended when we consider the stochastic integral with respect to spatial maximally distributed white noise. For stochastic integral with respect to temporal-spatial case, the integrand random fields can even contain
the ``non-adapted'' situation.}

\subsection{Stochastic integral with respect to the spatial maximally distributed white noise }

{We firstly define the stochastic integral with respect to the spatial maximally distributed white noise in a quite direct way. }

Let $\{{W}_{\gamma}\}_{\gamma\in\Gamma}$, $\Gamma=\mathcal{B}_{0}(\mathbb{R}^{d})$,
be a one-dimensional maximally distributed white noise defined on a complete sublinear
expectation space $({\Omega},\mathbb{L}^{1}(\Omega),\hat{E})$, with
$g(p)=\overline{\mu}p^{+}-\underline{\mu}p^{-}$, $-\infty<\underline{\mu}\leq\overline{\mu}<\infty.$
We introduce the following type of random fields, called simple random
fields.

Given $p\geq 1$, set
\begin{align*}
M_{g}^{p,0}(\Omega)=\lbrace\eta(x,\omega) & =\sum_{i=1}^{n}\xi_{i}(\omega)1_{A_{i}}(x),\;A_{1},\cdots,A_{n}\in\Gamma\text{ are mutually disjoint}\\
& i=1,2,\ldots,n,\ \ \xi_{1},\cdots,\xi_{n}\in \mathbb{L}^{p}(\Omega),\quad n=1,2,\cdots,\rbrace.
\end{align*}
For each simple random fields $\eta\in M_g^{p,0}(\Omega)$ of the form
\begin{equation}
\eta(x,\omega)=\sum_{i=1}^{n}\xi_{i}(\omega){\textbf{1}}_{A_{i}}(x),\label{Simple}
\end{equation}
 the related Bohner's integral for $\eta$ with respect to the Lebesgue measure $\lambda$ is
\[
I_B(\eta)=\int_{\mathbb{R}^{d}}\eta(x,\omega)\lambda(dx):=\sum_{i=1}^{n}\xi_{i}(\omega)\lambda_{A_{i}}.
\]
It is immediate that $I_B(\eta): M^{p,0}_g(\Omega)\mapsto \mathbb{L}^{p}(\Omega)$ is a linear and continuous mapping under
the norm for $\eta$, defined by,
 \[
\|\eta\|_{M^{p}}=\hat{E}[\int_{\mathbb{R}^{d}}|\eta({x},\omega)|^p\lambda(dx)]^{\frac{1}{p}}.
\]
The completion of $M_{g}^{p,0}(\Omega)$ under this norm is denoted by $M_{g}^{p}(\Omega)$ which is a Banach space.  The unique extension of the mapping $I_B$ is denoted by
\[
\int_{\mathbb{R}^{d}}\eta(x,\omega)\lambda(dx):=I_B(\eta),\,\, \eta\in M_{g}^{p}(\Omega).
\]

Now for a simple random field $\eta\in M^{p,0}_g(\Omega)$ of form (\ref{Simple}),
we  define its stochastic integral with respect to $W$ as
\[
I_W(\eta):=\int_{\mathbb{R}^{d}}\eta(x,\omega)W(dx)=\sum_{i=1}^{n}\xi_{i}(\omega)W_{A_{i}}.
\]
With this formulation, we have the following estimation.
\begin{lemma}\label{lemma2} For each $\eta\in M^{1,0}_g(\Omega)$ of form (\ref{Simple}), we have
\begin{equation}
\hat{E}\left[\left|\int_{\mathbb{R}^d}\eta(x,\omega)W(dx)\right|\right]\leq
\kappa\hat{E}\left[\int_{\mathbb{R}^d}|\eta(x,\omega)|\lambda(dx) \right]\label{ineq}
\end{equation}
where $\kappa=\max\{|\underline{\mu}|, |\overline{\mu}|\}$.

\end{lemma}

\begin{proof}
We have
\begin{align*}
\hat{E}[|\int_{\mathbb{R}^{d}}\eta(x,\omega)W(dx)|] & =\hat{E}[|\sum_{i=1}^{N}\xi_{i}(\omega)W_{A_{i}}|]\leq\hat{E}[\sum_{i=1}^{N}|\xi_{i}(\omega)|\cdot|W_{A_{i}}|]\\
 & \leq \kappa\hat{E}[\sum_{i=1}^{N}|\xi_{i}(\omega)|\cdot\lambda_{A_{i}}]=\kappa\hat{E}[\|\eta\|_{M_{g}^{1}(\Omega)}].
\end{align*}
The last inequality is due to the boundedness of maximal distribution (see Remark \ref{rem1}).
\end{proof}

This lemma shows that $I_W:M^{1,0}_g(\Omega)\mapsto \mathbb{L}^1(\Omega)$  is a linear continuous mapping. Consequently, $I_W$ can be uniquely extended to the whole domain $M^1_g(\Omega)$. We still denote
this extended mapping by
\[
\int_{\mathbb{R}^d}\eta W(dx):=I_W(\eta).
\]

{
\begin{remark}
Different from the stochastic integrals with respect to $G$-white noise in Ji and Peng \cite{PJ} which is only defined for the deterministic integrand, here the integrand can be a random field.
\end{remark}}

\subsection{Maximally distributed random fields of temporal-spatial types and related stochastic integral}

It is well-known that the framework of the classical  white noise
defined in  a probability space $(\Omega,\mathcal{F},P)$  with  $1$-dimensional temporal and $d$-dimensional spatial parameters is in fact a $\mathbb{R}^{1+d}$-indexed space type white noise.
{But  Peng \cite{Peng2011} and then  Ji and Peng \cite{PJ} observed a new phenomenon: Unlike the classical Gaussian white noise, the $d$-dimensional space-indexed $G$-white noise cannot have the property of incremental independence, thus spatial $G$-white noise is essentially different from temporal-spatial or temporal one. }
Things will become much direct  for the case of maximally distributed white noise due to the incremental  independence property of maximal distributions. This means that a time-space
maximally distributed $(1+d)$-white noise is essentially a $(1+d)$-spatial white noise. The corresponding
stochastic integral is also the same. But in order to make clear the dynamic properties,
 we still provide the description of the temporal-spatial white-noise on the time-space framework:
\begin{align*}
\mathbb{R}^+\times\mathbb{R}^d & =\{(t,x_{1},\ldots,x_{d})\in\mathbb{R}^+\times\mathbb{R}^d\},
 \end{align*}
where the index $t\in[0,\infty)$ is specially preserved to
be the index for time.

Let $\Gamma=\{A\in\mathcal{B}(\mathbb{R}^+\times\mathbb{R}^d), \lambda_A<\infty\}$, the maximally distributed white noise $\{W_{A}\}_{A\in\Gamma}$
is just like in the spatial case with dimension $1+d$.

More precisely, let
\begin{align*}
\Omega=\{\omega\in\mathbb{R}^{\Gamma}:\ \  & \omega(A\cup B)=\omega(A)+\omega(B),\\
 & \ \forall A,B\in\Gamma,\ \ A\cup B=\emptyset\},
\end{align*}
 and $W=({W}_{\gamma}(\omega)=\omega_{\gamma})_{\gamma\in\Gamma}$
the canonical random field.

For $T>0$, denote the temporal-spatial sets before time $T$ by
\begin{align*}
 \Gamma_{T}:=\{ A\in\Gamma: (s,x)\in A\Rightarrow 0\leq s<T\}.
\end{align*}
Set $\mathcal{F}_{T}=\sigma\{{W}_A, A\in\Gamma_{T}\}$,
$\mathcal{F}=\bigvee\limits _{T\geq0}\mathcal{F}_{T}$,
and
\begin{align*}L_{ip}(\Omega_{T})= & \{\varphi(W_{A_{1}},\ldots,W_{A_{n}}),\,\,\forall n\in\mathbb{N},\\
 & A_{i}\in\Gamma_{T},  i=1,\ldots,n, \varphi\in C_{b.Lip}(\mathbb{R}^{n}) \}.
\end{align*}

 We denote
\[
L_{ip}(\Omega)=\bigcup_{n=1}^{\infty}L_{ip}(\Omega_{n}).
\]

For each $X\in L_{ip}(\Omega)$, without loss of generality,
we assume $X$ has the form
\begin{align*}
X= & \varphi({W}_{A_{11}},\cdots,{W}_{{A}_{1m}},\cdots,{W}_{A_{n1}},\cdots,W_{A_{nm}}),
\end{align*}
where $A_{ij}=[t_{i-1},t_i)\times A_j$, $1\leq i\leq n, 1\leq j\leq m$, $0=t_0<t_{1}<\cdots<t_{n}<\infty$, $\{A_{1},\cdots,A_{m}\}\subset\mathcal{B}_{0}(\mathbb{R}^{d})$
are mutually disjoint and $\varphi\in C_{b.Lip}(\mathbb{R}^{n\times m})$.
Then the corresponding sublinear expectation for $X$ can be defined by

\begin{align*}
\hat{E}[X] & =\hat{E}[\varphi({W}_{A_{11}},\cdots,{W}_{A_{1m}},\cdots,{W}_{A_{n1}},\cdots,{W}_{A_{nm}})\\
 & =\max_{v_{ij}\in[{\underline{\mu}},\overline{\mu}]}\varphi(\lambda_{A_{11}}v_{11},\cdots,\lambda_{A_{1m}}v_{1m},\cdots,\lambda_{A_{n1}}v_{v_{n1},}\cdots,\lambda_{A_{nm}}v_{nm}),\\
 & \qquad \qquad \qquad \qquad \qquad \qquad \qquad \qquad \ \ \ \ \ \ \ \ \ \ \ \ \ \ \ 1\leq i\leq m,1\leq j\leq n
\end{align*}
and the related conditional expectation of $X$ under $\mathcal{F}_{t}$,
where $t_{j}\leq t<t_{j+1}$, denoted by $\hat{E}[X|\mathcal{F}_{t}]$,
is defined by
\begin{align*}
 &\hat{E}[\varphi({W}_{A_{11}},\cdots,{W}_{A_{1m}},\cdots,{W}_{A_{n1}},\cdots,{W}_{A_{nm}})|\mathcal{F}_{t}]\\
 =&\psi({W}_{A_{11}},\cdots,{W}_{A_{1m}},\cdots,{W}_{A_{j1}},\cdots,{W}_{A_{jm}}),
\end{align*}
where
\begin{align*}
\psi(x_{11},\cdots,x_{1m},\cdots,x_{j1},\cdots,x_{jm})=
\hat{E}[\varphi(x_{11},\cdots,x_{1m},\cdots,x_{j1},\cdots,x_{jm},\tilde{{W})}].
\end{align*}
Here
\[
\tilde{W}=({W}_{A_{(j+1)1}},\cdots,{W}_{A_{(j+1)m}},\cdots,{W}_{A_{n1}},\cdots,{W}_{A_{nm}}).
\]

 It is easy to verify that $\hat{E}[\cdot]$ defines a sublinear
expectation on $L_{ip}(\Omega)$ and the canonical process $({W}_{\gamma})_{\gamma\in\Gamma}$
is a one-dimensional temporal-spatial maximally distributed white noise on $(\Omega,L_{ip}(\Omega),\hat{E}).$

For each $p\geq1$, $T\geq 0$, we denote by ${L}_{g}^{p}(\Omega_T)$(resp.,
${L}_{g}^{p}(\Omega)$) the completion of $L_{ip}(\Omega_{T})$(resp.,
$L_{ip}(\Omega)$) under the norm $\|X\|_{p}:=(\hat{E}[|X|^{p}])^{1/p}$.
The conditional expectation $\hat{E}[\cdot\left|\mathcal{F}_{t}\right]:L_{ip}(\Omega)\rightarrow L_{ip}(\Omega_{t})$
is a continuous mapping under $\|\cdot\|_{p}$ and can be extended
continuously to the mapping ${L}_{g}^{p}(\Omega)\rightarrow{L}_{g}^{p}(\Omega_t)$
by
\[
|\hat{E}[X\left|\mathcal{F}_{t}\right]-\hat{E}[Y\left|\mathcal{F}_{t}\right]|\leq\hat{E}[|X-Y|\left|\mathcal{F}_{t}\right]\ \text{ for }X,Y\in L_{ip}(\Omega).
\]
It is easy to verify that the conditional expectation $\hat{E}[\cdot|\mathcal{F}_{t}]$
satisfies the following properties, and the proof is very similar to the corresponding one of Proposition 5.3 in Ji and Peng \cite{PJ}.

\begin{proposition}\label{pp1} For each
$t\geq 0$, the conditional expectation $\hat{E}[\cdot\left|\mathcal{F}_{t}\right]:{L}_{g}^{p}(\Omega)\rightarrow{L}_{g}^{p}(\Omega_t)$
satisfies the following properties: for any $X,Y\in{L}_{g}^{p}(\Omega)$,
$\eta\in{L}_{g}^{p}(\Omega_t)$,
\begin{description}
\item [{(i)}] $\hat{E}[X\left|\mathcal{F}_{t}\right]\geq\hat{E}[Y\left|\mathcal{F}_{t}\right]$\ for\ $X\geq Y$.
\item [{(ii)}] $\hat{E}[\eta\left|\mathcal{F}_{t}\right]=\eta$.
\item [{(iii)}] $\hat{E}[X+Y\left|\mathcal{F}_{t}\right]\leq\hat{E}[X\left|\mathcal{F}_{t}\right]+\hat{E}[Y\left|\mathcal{F}_{t}\right]$.
\item [{(iv)}] $\hat{E}[\eta X\left|\mathcal{F}_{t}\right]=\eta^{+}\hat{E}[X\left|\mathcal{F}_{t}\right]+\eta^{-}\hat{E}[-X\left|\mathcal{F}_{t}\right]$
if $\eta$ is bounded.
\item [{(v)}] $\hat{E}[\hat{E}[X\left|\mathcal{F}_{t}\right]|\mathcal{F}_{s}]=\hat{E}[X|\mathcal{F}_{t\wedge s}]$ for $s\geq0$.
\end{description}
\end{proposition}

Now we define the stochastic integral with respect to
the spatial-temporal maximally distributed white noise ${W}$, which is similar to the spatial situation.

For each given $p\geq1$, let ${M}^{p,0}(\Omega_T)$
be the collection of simple processes with the form:
\begin{equation}
f(s,x;\omega)=\sum\limits _{i=0}^{n-1}\sum_{j=1}^{m}X_{ij}(\omega)\textbf{1}_{A_{j}}(x)\textbf{1}_{[t_{i},t_{i+1})}(s),\label{simple-func}
\end{equation}
where $X_{ij}\in{L}_{g}^{p}(\Omega_T)$, $i=0,\cdots,n-1$,
$j=1,\cdots,m$, $0=t_{0}<t_{1}<\cdots<t_{n}=T$, and $\{A_{j}\}_{j=1}^{m}\subset\Gamma$
 is mutually disjoint.

\begin{remark}
Since we only require $X_{ij}\in{L}_{g}^{p}(\Omega_T)$, the integrand may ``non-adapted''. This issue is essentially different from the requirement of adaptability in the definition of stochastic integral with respect to temporal-spatial $G$-white noise in Ji and Peng \cite{PJ}.
\end{remark}

The completion of ${M}^{p,0}(\Omega_T)$ under the norm $\|\cdot\|_{{M}^{p}}$, denoted by ${M}_{g}^{p}(\Omega_T)$,
is a Banach space, where the Banach norm $\|\cdot\|_{{M}^{p}}$ is defined by
\begin{align*}
\|f\|_{{M}^{p}}:=&\left(\hat{E}\left[\int_{0}^{T}\int_{\mathbb{R}^{d}}|f(s,x)|^{p}ds\lambda(dx)\right]\right)^{\frac{1}{p}}\\=&\left\{ \hat{E}\left[\sum_{i=0}^{n-1}\sum_{j=1}^{m}|X_{ij}|^{p}(t_{i+1}-t_{i})\lambda_{A_{j}}\right]\right\} ^{\frac{1}{p}}.
\end{align*}

For $f\in{M}^{p,0}(\Omega_T)$ with the form as \eqref{simple-func}, the related stochastic integral with respect to the temporal-spatial maximally distributed white
noise ${W}$ can be defined as follows:
\begin{equation}
I_W(f)=\int_{0}^{T}\int_{\mathbb{R}^{d}}f(s,x){W}(ds,dx):=\sum\limits _{i=0}^{n-1}\sum_{j=1}^{m}X_{ij}{W}([t_j,t_{j+1})\times{A_{j}}).\label{Def-int}
\end{equation}

Similar to Lemma \ref{lemma2}, we have
\begin{lemma} \label{Wcontrol} For each $f\in{M}^{1,0}([0,T]\times\mathbb{R}^{d})$,
\begin{align}
\hat{E}\left[\left|\int_{0}^{T}\int_{\mathbb{R}^{d}}f(s,x){W}(ds,dx)\right|\right]\leq \kappa\hat{E}\left[\int_{0}^{T}\int_{\mathbb{R}^{d}}|f(s,x)|dsdx\right],\label{W continuity}
\end{align}
where $\kappa=\max\{|\overline{\mu}|,|\underline{\mu}|\}$.
\end{lemma}

Thus $I_W:{M}^{1,0}(\Omega_T)\mapsto {L}_{g}^{1}(\Omega_T)$ is a continuous linear mapping.
Consequently, $I_W$ can be uniquely extend to the domain ${M}_{g}^{1}(\Omega_T)$.
We still denote this mapping by

\[
\int_{0}^{T}\int_{\mathbb{R}^{d}}f(s,x){W}(ds,dx):=I_W(f)\ \ \text{for} \ f\in{M}_{g}^{1}(\Omega_T).
\]

\begin{remark}%We also note that \eqref{W continuity} holds for $f\in{M}_{g}^{1}(\Omega_T)$.
{Thanks to the boundedness of maximally distributed white noise, the domain of integrand $M_g^1(\Omega_T)$ is much larger since the usual requirement of adaptability for integrand can be dropped.}
\end{remark}

It is easy to check that the stochastic integral has the following
properties.

\begin{proposition} \label{stochastic integral properties} For each
$f,g\in{M}_{g}^{1}(\Omega_T)$, $0\leq s\leq r\leq t\leq T$,

\noindent (i) $\int_{s}^{t}\int_{\mathbb{R}^{d}}f(u,x){W}(du,dx)=\int_{s}^{r}\int_{\mathbb{R}^{d}}f(u,x){W}(du,dx)+\int_{r}^{t}\int_{\mathbb{R}^{d}}f(u,x){W}(du,dx).$

 \noindent(ii) $\int_{s}^{t}\int_{\mathbb{R}^{d}}(\alpha f(u,x)+g(u,x)){W}(du,dx)$\\
 $=\alpha\int_{s}^{t}\int_{\mathbb{R}^{d}}f(u,x){W}(du,dx)+\int_{s}^{t}\int_{\mathbb{R}^{d}}g(u,x){W}(du,dx)$,
where $\alpha\in{L}_{g}^{1}(\Omega_T)$ is bounded.

\end{proposition}

\begin{remark}
In particular, if we only consider temporal maximally distributed white noise and further assume that $\underline{\mu}\geq 0$. In this case, the index set $\Gamma=\{[s,t): 0\leq s<t<\infty\}$. The canonical process $W([0,t))$ is the quadratic variation process of $G$-Brownian motion, more details about the quadratic variation process can be found in Peng \cite{P2010}.
\end{remark}

\begin{acknowledgement}
This research is partially supported by NSF of China (No.L1624032, No.11526205, No.11601281), 111 Project of Chinese
SAFEA (No.B12023), National Key R\&D Program of
China (No.2018YFA0703900).
\end{acknowledgement}

%%%%%%%%%%%%%%%%%%%%%%%%%%%%%%%%%%%%%%%%%%%%%%%%%%%%%%%%%% Appendix sections. ????????????????, ??????????????%%%%%%%%%%%%%%%%%%%%%%%%%%%%%%%%%%%%%%%%%%%%%%%%%%%%%%%
\end{document}